\documentclass{amsart}
\usepackage{amsmath,amssymb,amsthm,amscd}
\newtheorem{theorem}{Theorem}[section]
\newtheorem{lemma}[theorem]{Lemma}
\newtheorem{proposition}[theorem]{Proposition}
\newtheorem{corollary}[theorem]{Corollary}

\newtheorem{remark}[theorem]{Remark}
\newtheorem{example}[theorem]{Example}

\begin{document}

\title{Gaussian trivial ring extensions and fqp-rings}
\author{Fran\c{c}ois Couchot}
\address{Universit\'e de Caen Basse-Normandie, CNRS UMR
  6139 LMNO,
F-14032 Caen, France}
\email{francois.couchot@unicaen.fr} 


\begin{abstract} Let $A$ be a commutative ring and $E$ a non-zero $A$-module. Neces\-sary and sufficient conditions are given for the trivial ring extension $R$ of $A$ by $E$ to be either arithmetical or Gaussian. The possibility for $R$ to be B\'ezout is also studied, but a response is only given in the case where $\mathrm{pSpec}(A)$ (a quotient space of $\mathrm{Spec}(A)$) is totally disconnected. Trivial  ring extensions which are fqp-rings are characterized only in the local case. To get a general result we intoduce the class of fqf-rings satisfying a weaker property than fqp-ring. Moreover, it is proven that the finitistic weak dimension of a fqf-ring is $0,\ 1$ or $2$ and its global weak dimension is $0,\ 1$ or $\infty$.
\end{abstract}

\maketitle


Trivial ring extensions are often used to give either examples or counterexamples of rings. One of the most famous is the chain ring $R$ which is not factor of a valuation domain (see \cite[X.6]{FuSa01} and \cite[Theorem 3.5]{FuSa85}). This ring $R$ is the trivial ring extension of a valuation domain $D$ by a non-standard uniserial divisible $D$-module. This example gives a negative answer to a question posed by Kaplansky.

In \cite{KaMa04}, \cite{KaMa11} and \cite{AbJaKa11} there are many results on trivial ring extensions and many examples of such rings. In particular, necessary and sufficient conditions are given for the trivial ring extension of a ring $A$ by an $A$-module $E$ to be either arithmetical or Gaussian in the following cases: either $A$ is a domain and $K$ is its quotient field, or $A$ is local and $K$ is its residue field, and $E$ is a $K$-vector space.

In our paper more general results are shown. For instance the trivial ring extension $R$ of a ring $A$ by a non-zero $A$-module $E$ is a chain ring if and only if $A$ is a valuation domain and $E$ a divisible module, and $R$ is Gaussian if and only if $A$ is Gaussian and $E$ verifies $aE=a^2E$ for each $a\in A$. Complete characterizations of arithmetical trivial ring extensions are given too. But B’ezout trivial ring extensions of a ring $A$ are characterized only in the case where $\mathrm{pSpec}(A)$ (a quotient space of $\mathrm{Spec}(A)$) is totally disconnected.

We also study trivial ring extensions which are fqp-rings. The class of fqp-rings was introduced in \cite{AbJaKa11} by Abuhlail, Jarrar and Kabbaj. We get a complete characterization of trivial ring extensions which are fqp-ring only in the local case. Each fqp-ring is locally fqp and the converse holds if it is coherent, but this is not generally true. We introduce the class of fqf-rings which satisfy the condition "each finitely generated ideal is flat modulo its annihilator", and it is exactly the class of locally fqp-rings. So, trivial fqf-ring extensions are completely characterized. This new class of rings contains strictly the class of fqp-rings. 

In \cite{Cou12} it is proven that each arithmetical ring has a finitistic weak dimension equal to $0,\ 1$ or $2$. We show that any fqf-ring satisfy this property too, and its global weak dimansion is $0,\ 1$ or $\infty$ as it is shown for fqp-rings in \cite{AbJaKa11}.

All rings in this paper are associative and commutative with unity, and all modules are
unital. We denote respectively $\mathrm{Spec}(A)$, $\mathrm{Max}(A)$ and $\mathrm{Min}(A),$ the space of prime ideals, maximal ideals and minimal prime ideals of
$A$, with the Zariski topology. If $I$ a subset of $R$, then we denote  
\[V(I) = \{ P\in\mathrm{Spec}(A)\mid I\subseteq P\}\quad \mathrm{and}\quad  D(I) =\mathrm{Spec}(A)\setminus V(I).\]

Let $A$ be a ring and $E$ an $A$-module. The {\bf trivial ring extension} of $A$ by $E$ (also called the idealization of $E$ over $A$) is the ring $R := A\propto E$ whose underlying group is $A\times E$ with multiplication given by $(a, e)(a', e') = (aa', ae' + a'e)$.

\section{Arithmetical rings}\label{S:ari}

An $R$-module $M$ is said to be \textbf{uniserial} if its set of submodules is totally ordered by inclusion and  $R$ is a \textbf{chain ring}\footnote{we prefer ``chain ring '' to ``valuation ring'' to avoid confusion with ``Manis valuation ring''.} if it is uniserial as $R$-module. Recall that a chain ring $R$ is said to be {\bf Archimedean} if its maximal ideal is the sole non-zero prime ideal.

\begin{proposition}\label{P:chain}
Let $A$ be a ring, $E$ a non-zero $A$-module and $R=A\propto E$ the trivial ring extension of $A$ by $E$. The following two conditions are equivalent:
\begin{enumerate}
\item $R$ is a chain ring;
\item $A$ is a valuation domain and $E$ is a uniserial divisible module.
\end{enumerate}
\end{proposition}
\begin{proof}
It is well known that $(2)\Rightarrow (1)$ (for instance, see \cite[Example I.1.9]{FuSa85}). Conversely, $A$ is a chain ring because it is a factor of $R$, and $E$ is uniserial because it is isomorphic to an ideal of $R$. Let $0\ne a\in A$ and $x\in E$. It is obvious that $(0,x)\in R(a,0)$. So, $(0,x)=(a,0)(b,y)$ for some $y\in E$ and $b\in A$, whence $x=ay$. Hence $E=aE$. Let $a,\ b$ be non-zero elements of $A$. Then $E=aE=abE$. We deduce that $ab\ne 0$. Hence $A$ is a domain and $E$ is divisible.
\end{proof}

Let $E$ be a module over a ring $R$. We say that $E$ has a {\bf distributive} lattice of submodules if $E$ satisfies one of the following two equivalent conditions:
\begin{itemize}
\item $(M+N)\cap P=(M\cap P)+(N\cap P)$ for any submodules $M,\ N,\ P$ of $E$;
\item $(M\cap N)+P=(M+P)\cap (N+P)$ for any submodules $M,\ N,\ P$ of $E$.
\end{itemize}

The following proposition can be proved as \cite[Theorem 1]{Jen66}.
\begin{proposition}\label{P:dist}
Let $E$ be a module over a ring $R$. The following conditions are equivalent:
\begin{enumerate}
\item $E$ has a distributive lattice of submodules;
\item $E_P$ is a uniserial module for each maximal ideal $P$ of $R$.
\end{enumerate}
\end{proposition}

A ring (respectively domain) $R$ is said to be {\bf arithmetical} (respectively {\bf Pr\"ufer}) if its lattice of ideals is distributive. An $R$-module $E$ is {\bf FP-injective} if, for each finitely presented $R$-module $F$, $\mathrm{Ext}_R^1(F,E)=0$. When $R$ is a Pr\"ufer domain  then a module is FP-injective if and only if it is divisible.

\begin{lemma}
\label{L:loc}Let $A$ be a ring, $E$ a non-zero $A$-module and $R=A\propto E$ the trivial ring extension of $A$ by $E$. Let $S$ be a multiplicative subset of $R$ and $S'$ the image of $S$ by the map $R\rightarrow A$ defined by $(a,x)\mapsto a$. Then $S^{-1}R=S'^{-1}A\propto S'^{-1}E$.
\end{lemma}
\begin{proof}
If $(a,x)\in R$ and $(s,y)\in S$ then it is easy to check that
\[\frac{(a,x)}{(s,y)}=(\frac{a}{s},\frac{sx-ay}{s^2})=\frac{(sa,sx-ay)}{s^2}=\frac{(a,x)(s,-y)}{s^2}.\]
\end{proof}

For any module $E$ over ring $R$ we define $\mathrm{Supp}(E)$ as:
\[\mathrm{Supp}(E)=\{P\mid P\ \mathrm{prime\ ideal\ such\ that}\ E_P\ne 0\}.\]
An $R$-module $E$ is {\bf locally FP-injective} if for each maximal ideal $P$, $E_P$ is FP-injective.

From Propositions \ref{P:chain} and \ref{P:dist} and Lemma \ref{L:loc} we deduce the following corollary.
\begin{corollary}\label{C:arith}
Let $A$ be a ring, $E$ a non-zero $A$-module and $R=A\propto E$ the trivial ring extension of $A$ by $E$. The following two conditions are equivalent:
\begin{enumerate}
\item $R$ is arithmetical;
\item $A$ is arithmetical, $A_P$ is an integral domain for each $P\in\mathrm{Supp}(E)$, and $E$ is locally FP-injective and has a distributive lattice of submodules.
\end{enumerate}
\end{corollary}

\begin{remark}
Let us observe that $E$ is a module over $A/N$ where $N$ is the nilradical of $A$ if $R$ is arithmetical.
\end{remark}
\begin{proof}
Let $a\in N$ and $P$ be a maximal ideal of $A$. Either $P\in\mathrm{Supp}(E)$ and $aR_P=0$, or $E_P=0$. So, $aE=0$.
\end{proof}

\section{Gaussian rings}
\label{S:Gau}

Let $R$ be a ring. For a polynomial $f\in R[X]$, denote  by $c(f)$ (the content of $f$) the ideal of $R$ generated by the coefficients of $f$. We say that $R$ is {\bf Gaussian} if $c(fg)=c(f)c(g)$ for any two polynomials $f$ and $g$ in $R[X]$. By \cite{Tsa65}, a local ring $R$ is Gaussian if and only if, for any ideal $I$ generated by two elements $a,b,$ in $R$, the   following two properties hold:
\begin{enumerate}
\item $I^2$ is generated by  $a^2$ or $b^2$;
\item if $I^2$ is generated by $a^2$ and $ab=0$, then $b^2=0$.
\end{enumerate}

\begin{proposition}\label{P:localGauss}
Let $A$ be a local ring, $E$ a non-zero $A$-module and $R=A\propto E$ the trivial ring extension of $A$ by $E$. The following two conditions are equivalent:
\begin{enumerate}
\item $R$ is Gaussian;
\item $A$ is Gaussian and $aE=a^2E$ for each $a\in A$.
\end{enumerate}
\end{proposition}
\begin{proof} $(1)\Rightarrow (2)$. As factor of $R$, $A$ is Gaussian. Let $0\ne a\in A$ and $x\in E$. First assume $a^2\ne 0$. Then $(a,0)^2\ne 0$ and $(0,x)^2=0$. So, $(a,0)(0,x)=(0,ax)=(0,y)(a^2,0)=(0,a^2y)$ for some $y\in E$. We get $a^2y=ax$. Now assume $a^2=0$. In this case $(a,0)^2=0$. It follows that $(a,0)(0,x)=0$, whence $ax=0$. In the two cases $aE=a^2E$.

$(2)\Rightarrow (1)$. Let $r=(a,x)$ and $s=(b,y)$ two elements of $R$. Since $A$ is Gaussian we may assume that $ab=ca^2$ and $b^2=da^2$ for some $c,\ d\in A$. We shall prove that $(Rr+Rs)^2=Rr^2$ by showing there exist $z$ and $v$ in $E$ satisfying the following two equations:
\begin{equation}\label{EQ:G1}
(a,x)(b,y)=(c,z)(a,x)^2
\end{equation}
\begin{equation}\label{EQ:G2}
(b,y)^2=(d,v)(a,x)^2
\end{equation}
For (\ref{EQ:G1}) we only need to show that $a^2z=ay+bx-2acx$. Since $bE=b^2E$ then $bx=b^2m=da^2m$ for some $m\in E$. So, $ay+bx-2acx=a(y+adm-2cx)$. Now we use the equality $aE=a^2E$ to conclude.

For (\ref{EQ:G2}) we only need to show that $a^2v=2by-2adx$. Since $bE=b^2E$ then $by=b^2n=da^2n$ for some $n\in E$. So, $2by-2adx=a(2adn-2dx)$. Now we use the equality $aE=a^2E$ to conclude.

If $rs=0$ then $ab=0$ and $b^2=0$. So, $s^2=(0,2by)$. But $bE=b^2E=0$, whence $by=0$ and $s^2=0$.
\end{proof}

\begin{corollary}\label{C:Gauss}
Let $A$ be a ring, $E$ a non-zero $A$-module and $R=A\propto E$ the trivial ring extension of $A$ by $E$. The following two conditions are equivalent:
\begin{enumerate}
\item $R$ is Gaussian;
\item $A$ is Gaussian and $aE=a^2E$ for each $a\in A$.
\end{enumerate}
\end{corollary}
\begin{proof}
We deduce this corollary from Proposition \ref{P:localGauss} and from the fact that a ring $A$ is Gaussian if and only $A_P$ is Gaussian for each maximal ideal $P$ and that $aE=a^2E$ if and only if $aE_P=a^2E_P$ for each maximal ideal $P$.
\end{proof}

\begin{remark}\label{R:Gauss}
\textnormal{Let $A$ be a Gaussian ring.  Let $E$ be an $A$-module such that $aE=a^2E$ for each $a\in A$. If $a\in N$, where $N$ is the nilradical of $A$, it is easy to see that $aE=0$. So, $E$ is a module over $A/N$ which is arithmetical by \cite[Corollary 7]{Luc08}.}
\end{remark}

\begin{proposition}
Let $A$ be an integral domain and $E$ an $A$-module such that $aE=a^2E$ for each $a\in A$. Then:
\begin{enumerate}
\item $E$ is divisible if it is torsion-free;
\item when $A$ is local, $E$ is semisimple if it is finitely generated;
\item if $A$ is an Archimedean valuation domain, then $E$ is the extension of a divisible module by a semisimple module.
\end{enumerate}
\end{proposition}
\begin{proof}
$(1)$. Let $0\ne a\in A$ and $x\in E$. There exists $y\in E$ such that $ax=a^2y$. Since $E$ is torsion-free we get $x=ay$.

$(2)$. Let $P$ be the maximal ideal of $A$ and $a\in P$. We have $aE=a(aE)$. By Nakayama Lemma $aE=0$. So, $E$ is an $A/P$-module.

$(3)$. To do this we show that $PE$ is divisible. Let $x\in PE$ and $0\ne s\in A$. Then $x=ay$ for some $a\in P$ and $y\in E$. It is easy to check that $\cap_{n\in\mathbb{N}}Aa^n$ is a prime ideal. So, this intersection is $0$. There exists an integer $n$ such that $s\notin Aa^n$. We get that $aE=saE=a^{n+1}E$. Hence $x\in saE\subseteq sPE$. 
\end{proof}

\begin{example}
Let $A$ be a valuation domain, $\Lambda$ an index set and $(L_{\lambda})_{\lambda\in\Lambda}$ a family of prime ideals. For each $\lambda\in\Lambda$ let $E_{\lambda}$ be a non-zero divisible $A/L_{\lambda}$-module. Let $E_1=\oplus_{\lambda\in\Lambda}E_{\lambda}$ and $E_2=\prod_{\lambda\in\Lambda}E_{\lambda}$. Then $aE_i=a^2E_i$ for each $a\in A$ and for $i=1,\ 2$.
\end{example}
\begin{proof}
Let $0\ne a\in A$ and $x=(x_{\lambda})_{\lambda\in\Lambda}\in E_i$ where $i=1,\ 2$. If $ax_{\lambda}=0$ then $ax_{\lambda}=a^2y_{\lambda}$ with $y_{\lambda}=0$. If $ax_{\lambda}\ne 0$ then $a\notin L_{\lambda}$. Since $E_{\lambda}$ is divisible over $A/L_{\lambda}$ there exists $y_{\lambda}\in E_{\lambda}$ such that $ax_{\lambda}=a^2y_{\lambda}$. So, $ax=a^2y$ with $y=(y_{\lambda})_{\lambda\in\Lambda}$.
\end{proof}

Let $A$ be a valuation domain. A non-zero prime ideal $L$ that is not the union of prime ideals properly contained in it is called {\bf branched}.

\begin{proposition}\label{P:dis}
Let $A$ be a valuation domain. Assume that each non-zero prime ideal is branched. For any $A$-module $E$ the following conditions are equivalent:
\begin{enumerate}
\item $aE=a^2E$ for each $a\in A$;
\item for each prime ideal $L$, $LE/L'E$ is an $A/L'$-module divisible, where $L'$ is the union of all prime ideals properly contained in $L$.
\end{enumerate}
\end{proposition}
\begin{proof}
$(1)\Rightarrow (2)$. Let $s\in L\setminus L'$ and $x\in LE$. Then $x=ay$ for some $a\in L$ and $y\in E$. We may assume that $a\notin L'$. So, $L'=\cap_{n\in\mathbb{N}}Aa^n$ and there exists $n\in\mathbb{N}$ such that $s\notin Aa^n$. Whence $aE=a^{n+1}E\subseteq saE\subseteq aE$. Hence $x\in sLE$.

$(2)\Rightarrow (1)$. Let $0\ne a$ a non-unit of $A$. Let $L$ be the prime ideal which is the intersection of all prime ideals containing $a$. Then $a\notin L'$. Let $x=ay\in aE$. Since $LE/L'E$ is divisible over $A/L'$, then $x\in a^2E+L'E\subseteq a^2E$.
\end{proof}

A chain ring is said to be {\bf strongly discrete} if it contains no idempotent prime ideal. Each non-zero prime ideal of a strongly discrete valuation domain is branched. 

\section{B\'ezout rings}\label{S:Bez}

A ring is a {\bf B\'ezout ring} if every finitely generated ideal is principal. A ring $R$ is {\bf Hermite} if $R$ satisfies the following property~: for every $(a,b) \in R^2$, there exist $d, a', b'$ in $R$ such that $a = da'$, $b = db'$ and $Ra' + Rb' = R$. We say that $R$ is an {\bf elementary divisor ring} if for every matrix $A$,
with entries in $R$, there  exist a diagonal matrix $D$ and invertible matrices $P$ and $Q$, with entries in $R$, such that $PAQ = D$. Then we have the following
implications:

\centerline{elementary divisor ring $\Rightarrow$ Hermite ring
$\Rightarrow$ B\'ezout ring $\Rightarrow$ arithmetical ring;}
but these implications are not reversible: see \cite{GiHen56} or \cite{Car87}.

\begin{proposition}\label{P:Bez}
Let $A$ be a ring and $N$ its nilradical. Assume that $N$ is prime. Let $E$ be a non-zero $A$-module and $R=A\propto E$ the trivial ring extension of $A$ by $E$. The following three conditions are equivalent:
\begin{enumerate}
\item $R$ is Hermite;
\item $R$ is B\'ezout;
\item $A$ is B\'ezout, $A_P$ is a domain for each $P\in\mathrm{Supp}(E)$, $E$ is FP-injective and all its finitely generated submodules are cyclic.
\end{enumerate}
\end{proposition}
\begin{proof}
$(1)\Leftrightarrow (2)$. It is well known that each Hermite ring is B\'ezout. Since $R$ contains a unique minimal prime ideal, then the converse holds by \cite[Theorem 2]{Hen55}. 

$(2)\Rightarrow (3)$. By Corollary \ref{C:arith} $E$ is FP-injective and $A_P$ is a domain for each $P\in\mathrm{Supp}(E)$ . As factor of $R$ $A$ is B\'ezout, and each finitely generated submodule of $E$ is cyclic because $E$ is isomorphic to an ideal of $R$.

$(3)\Rightarrow (2)$. Let $a\in N$ and $x\in E$. For each $P\in\mathrm{Supp}(E)$ $aA_P=0$ (since $A_P$ is a domain). It follows that $(0:a)\nsubseteq P$ and consequently $\mathrm{Supp}(E)\subseteq D((0:a))$. For each $P\notin\mathrm{Supp}(E)$, $A_Px=0$. It follows that $(0:x)\nsubseteq P$ and consequently $D((0:a))\cup D((0:x))=\mathrm{Spec}(A)$. So, $(0:a)+(0:x)=A$, whence there exist $b\in (0:a)$ and $c\in (0:x)$ such that $b+c=1$. Then $a=ca,\ x=bx$ and it is easy to check that $R(a,0)+R(0,x)=R(a,x)$. Let $(a,x)$ and $(b,y)$ be two elements of $R$. If $(a,b)\in N\times N$ then there exist $d\in A$ such that $Ad=Aa+Ab$, and $z\in E$ such that $Ax+Ay=Az$. It follows that $R(a,x)+R(b,y)=R(d,z)$. Now assume that $(a,b)\notin N\times N$. There exist $d,a',b',s,t\in A$ such that $a=da',\ b=db'$ and $sa+tb=d$. Let $z=sx+ty$. Since $E$ is divisible over $A/N$ and $d\notin N$, there exist $x',\ y'\in E$ such that $x-a'z=dx'$ and $y-b'z=dy'$. Now it is easy to check that $(a,x)=(d,z)(a',x')$, $(b,y)=(d,z)(b',y')$ and $(s,0)(a,x)+(t,0)(b,y)=(d,z)$.
\end{proof}

If $R$ is a ring, we consider on $\mathrm{Spec}(R)$ the equivalence relation $\mathcal{R}$ defined by   $L\mathcal{R} L'$ if there exists a finite sequence of prime ideals $(L_k)_{1\leq k\leq n}$ such that $L=L_1,$ $L'=L_n$ and $\forall k,\ 1\leq k\leq (n-1),$ either $L_k\subseteq L_{k+1}$ or $L_k\supseteq L_{k+1}$. We denote by $\mathrm{pSpec}(R)$ the quotient space of $\mathrm{Spec}(R)$ modulo $\mathcal{R}$ and by $\lambda_R: \mathrm{Spec}(R)\rightarrow\mathrm{pSpec}(R)$ the natural map. The quasi-compactness of $\mathrm{Spec}(R)$ implies the one of $\mathrm{pSpec}(R)$, but generally $\mathrm{pSpec}(R)$  is not  $T_1$: see \cite[Propositions 6.2 and 6.3]{Laz67}.  A topological space is called \textbf{totally disconnected} if each of its connected components contains only one point. Every Hausdorff topological space $X$ with a base of clopen neighbourhoods is totally disconnected and the converse holds if $X$ is compact (see \cite[Theorem 16.17]{GiJe60}).

An ideal $I$ of a ring $A$ is {\bf pure} if and only if $A/I$ is a flat $A$-module.

\begin{theorem}\label{T:disconn}
Let $A$ be a  ring, $E$ a non-zero $A$-module and $R=A\propto E$ the trivial ring extension of $A$ by $E$. Assume that $\mathrm{pSpec}(A)$ is totally disconnected. The following three conditions are equivalent:
\begin{enumerate}
\item $R$ is Hermite;
\item $R$ is B\'ezout;
\item $A$ is is B\'ezout, $A_P$ is a domain for each $P\in\mathrm{Supp}(E)$, $E$ is locally FP-injective and all its finitely generated submo\-dules are cyclic.
\end{enumerate}
\end{theorem}
\begin{proof}
$(1)\Rightarrow (2)$ is well known and we show $(2)\Rightarrow (3)$ as in Proposition \ref{P:Bez}.

$(3)\Rightarrow (1)$. By \cite[Proposition 2.2]{Couc09} $\mathrm{pSpec}(A)$ is compact. Since $A=R/J$ where $J$ is contained in the nilradical of $R$ then $\mathrm{pSpec}(A)$ and $\mathrm{pSpec}(R)$ are homeomorphic. Let $x\in\mathrm{pSpec}(R)$, $I(x)$ the pure ideal of $R$ for which $x=V(I(x))$ and $I'=I(x)/NI(x)$ (see \cite[Lemma 2.5]{Couc09}). Thus $I'$ is a pure ideal of $A$ and it is contained in only one minimal prime ideal. So, the nilradical of $A/I'$ is prime. If $S=1+I(x)$ then $S^{-1}R=R/I(x)$. By using Lemma \ref{L:loc} we get that $R/I(x)=A/I'\propto E/I'E$. By Proposition \ref{P:Bez} $R/I(x)$ is B\'ezout for each $x\in\mathrm{pSpec}(R)$. We conclude that $R$ is Hermite by \cite[Theorem 3.1]{Couc09}.
\end{proof}

Recall that a ring $A$ is {\bf coherent} (respectively {\bf semihereditary}) if all its finitely generated ideals are finitely presented (respectively projective).

\begin{corollary}\label{C:semih}
Let $A$ be a coherent reduced ring, $E$ a non-zero $A$-module and $R=A\propto E$ the trivial ring extension of $A$ by $E$. The following three conditions are equivalent:
\begin{enumerate}
\item $R$ is Hermite;
\item $R$ is B\'ezout;
\item $A$ is is B\'ezout, $E$ is FP-injective and all its finitely generated submodules are cyclic.
\end{enumerate}
\end{corollary}
\begin{proof}
Since $A$ is B\'ezout, reduced and coherent then $A$ is semihereditary. By \cite[Proposition 10]{Que71} $\mathrm{Min}(A)$ is compact, and by \cite[Proposition 2.2]{Couc09} $\mathrm{pSpec}(A)$ is homeomorphic to $\mathrm{Min}(A)$, and consequently it is totally disconnected. Since $A$ is coherent, each FP-injective module is locally FP-injective. We conclude by Theorem \ref{T:disconn}.
\end{proof}

\begin{corollary}
Let $A$ be a ring, $E$ a non-zero $A$-module and $R=A\propto E$ the trivial ring extension of $A$ by $E$. Assume that $\mathrm{pSpec}(A)$ is totally disconnected. The following two conditions are equivalent:
\begin{enumerate}
\item $R$ is an elementary divisor ring;
\item $A$ is an elementary divisor ring, $A_P$ is a domain for each $P\in\mathrm{Supp}(E)$, $E$ is locally FP-injective and all its finitely generated submodules are cyclic.
\end{enumerate}
\end{corollary}
\begin{proof}  It is a consequence of Theorem \ref{T:disconn} and the following.
By \cite[Theorem 6]{GiHe56} a ring $R$ is an elementary divisor ring if and only if $R$ is Hermite and for any $a,b,c\in R$ such that $Ra+Rb+Rc=R$ there exist $p,q\in R$ such that $Rpa+R(pb+qc)=R$. By \cite[Theorem 3]{Hen55} a Hermite ring $R$ is an elementary divisor ring if and only if so is $R/N$, where $N$ is the nilradical of $R$.
\end{proof}

\begin{corollary}\label{C:Her}
Let $A$ be a coherent reduced B\'ezout ring, $E$ a non-zero FP-injective $A$-module. Assume that each finitely generated submodule of $E$ is cyclic. Then, for any $x,y\in E$ there exist $z\in E$ and an invertible $2\times 2$ matrix $B$ with coefficients in $A$ such that $\displaystyle{\binom{z}{0}=B\binom{x}{y}}$.
\end{corollary}
\begin{proof}
Let $R=A\propto E$. Let $x,y\in A$. By Corollary \ref{C:semih} $R$ is Hermite. So, there exist $r\in R$ and an invertible $2\times 2$ matrix $C$ with entries in $R$ such that $\displaystyle{\binom{r}{0}=C\binom{(0,x)}{(0,y)}}$. It is obvious that $r=(0,z)$ for some $z\in E$. From $C$ we deduce an invertible $2\times 2$ matrix $B$ with entries in $A$ satisfying $\displaystyle{\binom{z}{0}=B\binom{x}{y}}$.
\end{proof}

Contrary to the two first sections we do not get general results, even in the case where $A$ is reduced. In \cite{Vas76} there are two examples of reduced B\'ezout rings which are not semihereditary. For the first example $A$ (\cite[Example 1.3b]{Vas76}) $\mathrm{pSpec}(A)$ is totally disconnected, so, if $E$ is an $A$-module satisfying the conditions of Theorem \ref{T:disconn} then $A\propto E$ is B\'ezout. But for the second example $A$ (\cite[Example 6.2]{Vas76}), $\mathrm{pSpec}(A)$ is connected and infinite, so, we do not know if $A$ admits B\'ezout proper trivial ring extensions.

\section{fqp-rings and fqf-rings}\label{S:fqp}

Let $A$ be a ring, $M$ an $A$-module. An $A$-module V is {\bf $M$-projective} if the natural homomorphism $\mathrm{Hom}_A(V ,M)\rightarrow\mathrm{Hom}_A (V , M /X)$ is surjective for every submodule $X$ of $M$. We say that  $V$ is {\bf quasi-projective} if $V$ is $V$-projective. A ring $A$ is said to be an {\bf fqp-ring} if every finitely generated ideal of $A$ is quasi-projective.

\begin{theorem}
\label{T:localfqp} Let $A$ a local ring and $N$ its nilradical. Then $A$ is an fqp-ring if and only if either $A$ is a chain ring or $A/N$ is a valuation domain and $N$ is a divisible torsionfree $A/N$-module.
\end{theorem}
\begin{proof}
Assume that $A$ is an fqp-ring but not a chain ring. By  \cite[Lemmas 3.12 and 4.5]{AbJaKa11}, $N^2=0$ and every zero-divisor belongs to $N$. So, $N$ is prime. From  \cite[Lemma 3.8]{AbJaKa11} it follows that any two elements of $A$ which are not in $N$ are comparable. This implies that $A/N$ is a valuation domain. Let $a\notin N$ and $b\in N$. By using again  \cite[Lemma 3.8]{AbJaKa11} we get that $b\in Aa$. Hence $N$ is divisible over $A/N$ and it is torsionfree since each element in $A\setminus N$ is regular.

Conversely, it is easy to see that each chain ring is fqp. So, we may assume that $A$ is not a chain ring. Let $I$ a finitely generated ideal of $A$. If $I\subseteq N$ then $I$ is a free module over $A/N$. Consequently $I$ is quasi-projective. Now, suppose that $I\nsubseteq N$. Thus $(I+N)/N$ is principal, and from the fact that $N$ is divisible over $A/N$ we deduce that $I$ is principal too. Hence $I$ is quasi-projective.  
\end{proof}

\begin{theorem}
\label{T:trifqp} Let $A$ be a local ring, $N$ its nilradical, $E$ a non-zero $A$-module and $R=A\propto E$ the trivial ring extension of $A$ by $E$. Then $R$ is an fqp-ring if and only if $A$ is an fqp-ring and  one of the following conditions holds:
\begin{enumerate}
\item $A$ is a valuation domain and $E$ is divisible and uniserial;
\item $E$ is divisible and torsionfree over $A/N$, each zero-divisor of $A$ belongs to $N$ and $N^2=0$.
\end{enumerate}
\end{theorem}
\begin{proof}
First assume that $R$ is fqp. By \cite[Proposition 4.1]{AbJaKa11} so is $A$. If $R$ is a chain ring we use Proposition \ref{P:chain}. If not, since $R$ is Gaussian by \cite[Theorem 3.2]{AbJaKa11} then $E$ is an $A/N$-module by Remark \ref{R:Gauss}. Let $N'$ be the nilradical of $R$. It is easy to see that $R/N'\cong A/N$ and $N'\cong N\oplus E$. By Theorem \ref{T:localfqp}, $N'^2=0$, $N^2=0$ and $E$ is divisible and torsionfree over $A/N$.

Conversely, condition (1) implies that $R$ is a chain ring. Now, assume that condition (2) holds. As in the first part of the proof,  $R/N'\cong A/N$, $N'\cong N\oplus E$ and $N'^2=0$. So, $N'$ is divisible and torsionfree over $R/N'$. We conclude by Theorem \ref{T:localfqp}.
\end{proof}

\begin{corollary}
Let $A$ be a valuation domain which is not a field, $E$ a non-zero divisible $A$-module and $R=A\propto E$. Then:
\begin{enumerate}
\item if $E$ is uniserial then $R$ is a chain ring;
\item if $E$ is torsionfree but not uniserial then $R$ is an fqp-ring which is not arithmetical;
\item if $E$ is neither torsionfree nor uniserial then $R$ is Gaussian but not an fqp-ring. 
\end{enumerate}
\end{corollary}

\begin{proposition}\label{P:coh}
Let $A$ be a coherent ring. Then $A$ is an fqp-ring if and only so is $A_P$ for each maximal ideal $P$ of $A$.
\end{proposition}
\begin{proof}
Let $I$ be a finitely generated ideal. Then $(\mathrm{End}(I))_P\cong\mathrm{End}(I_P)$ since $I$ is finitely presented.  We use \cite[Lemma 3.7]{AbJaKa11} to conclude.
\end{proof}

We say that a ring $A$ is a {\bf fqf-ring} if each finitely generated ideal $I$ is flat modulo its annihilator, i.e $I$ is a flat $A/(0:I)$-module.  It is obvious that every fqp-ring $A$ is a fqf-ring (by \cite[Lemma 2.2]{AbJaKa11} a finitely generated module is quasi-projective if and only if it is projective modulo its annihilator). By Proposition \ref{P:coh} the converse holds if $A$ is coherent. But we shall see that it is not generally true. 

\begin{corollary}\label{C:fqf}
Let $A$ be a ring, $N$ its nilradical, $E$ a non-zero $A$-module and $R=A\propto E$. Then the following conditions are equivalent:
\begin{enumerate}
\item $R$ is a fqf-ring;
\item $A$ is a fqf-ring and for each $P\in\mathrm{Supp}(E)$ either $N_P=0$ and $E_P$ is a uniserial divisible $A_P$-module or $N_P^2=0$ and $E_P$ and $N_P$ are divisible and torsionfree over $A_P/N_P$.
\end{enumerate}
\end{corollary}
\begin{proof}
Over a local ring each finitely generated flat module is free. So, a local ring is a fqf-ring if and only if it is an fqp-ring (by \cite[Lemma 2.2]{AbJaKa11}). As a module is flat if and only if it is locally flat, then a ring $A$ is a fqf-ring if and only if so is $A_P$ for each maximal ideal $P$ of $A$. So, we use Theorem \ref{T:trifqp} to conclude.
\end{proof}

\begin{example}
Let $A$ be a von Neumann regular ring which is not self-injective, $H$ the injective hull of $A$, $x\in H\setminus A$, $E=A+Ax$ and $R=A\propto E$. Then $R$ is a fqf-ring which is not an fqp-ring. 
\end{example}
\begin{proof}
Clearly $A$ is a reduced fqp-ring, and, for each maximal ideal $P$, $E_P$ is injective and torsionfree over the field $A_P$. By corollary \ref{C:fqf} $R$ is a fqf-ring. On the other hand $E$ is isomorphic to a finitely generated ideal $J$ of $R$. It is obvious that $E$ is a faithful $A$-module, so its annihilator as $R$-module is $J$. If $P\in\mathrm{Supp}(E/A)$ then $E_P$ is of rank $2$ and if  $P\notin\mathrm{Supp}(E/A)$ it is of rank one. But $(A:x)$ is an essential ideal of $A$ and consequently it is not generated by an idempotent. So, $\mathrm{Supp}(E/A)=V(\mathrm{ann}(E/A))=V((A:x))$  is not open . It follows that the map $\mathrm{Spec}(A)\rightarrow\mathbb{N}$ defined by $P\mapsto\mathrm{rank}_{A_P}(E_P)$ is not locally constant. Hence $E$ is not a projective $A$-module by \cite[Th\'eor\`eme II.\S 5.2.1]{Bou61}. By \cite[Lemma 2.2]{AbJaKa11} it is not quasi-projective over $R$, whence $R$ is not an fqp-ring. 
\end{proof}

\section{Finitistic and global weak dimensions of fqf-rings}

Let $R$ be a ring. If $M$ is an $R$-module, we denote by $\mathrm{w.d.}(M)$ its {\bf weak dimension}. Recall that $\mathrm{w.d.}(M)\leq n$ if $\mathrm{Tor}^R_{n+1}(M,N)=0$ for each $R$-module $N$. For any ring $R$, its {\bf global weak dimension} $\mathrm{w.gl.d}(R)$ is the supremum of $\mathrm{w.d.}(M)$ where $M$ ranges over all (finitely presented cyclic) $R$-modules. Its {\bf finitistic weak dimension}
$\mathrm{f.w.d.}(R)$ is the supremum of $\mathrm{w.d.}(M)$ where $M$ ranges over all  $R$-modules of finite weak dimension. 

We shall extend \cite[Theorem 3.11]{AbJaKa11} and\cite[Theorem 1]{Cou12} to fqf-rings by showing the following theorem:
\begin{theorem} The global weak dimension of a fqf-ring is $0,\ 1$ or $\infty$ and its finitistic weak dimension is $0,\ 1$ or $2$.
\end{theorem}
\begin{proof}
The first assertion can be proven as \cite[Theorem 3.11]{AbJaKa11}.

Let $R$ be a fqf-ring. It is well known that an $R$-module $M$ is flat if and only if so is $M_P$ for each $P\in\mathrm{Max}\ R$. So, we may assume that $R$ is local. By Theorem \ref{T:trifqp} either $R$ is a chain ring and we use \cite[Theorem 2]{Cou12} or $R$ satisfies the assumptions of the following proposition and we conclude by using it. 
\end{proof}
\begin{proposition}
Let $R$ be a local ring and $N$ its nilradical. Assume that $R/N$ is a valuation domain, $N^2=0$ and $N$ is divisible and torsionfree over $R/N$. Then:
\begin{enumerate}
\item if $N$ is the maximal ideal then $\mathrm{f.w.d.}(R)=0$;
\item if $N$ is not maximal then $\mathrm{f.w.d.}(R)=1$.
\end{enumerate}
\end{proposition}
\begin{proof}
$(1)$. In this case $R$ is a primary ring. We conclude by \cite[Theorems P and 6.3]{Bas60}.

$(2)$. Let $Q$ be the quotient ring of $R$. Since each element of $R$ which is not in $N$ is regular then $Q=R_N$ and since $N$ is divisible and torsionfree over $R/N$ then $N$ is a $Q$-module and it is the maximal ideal of $Q$. From $(1)$ we deduce that $\mathrm{f.w.d.}(Q)=0$. 

Let $M$ be an $R$-module with $\mathrm{w.d.}(M)<\infty$. Then we have $\mathrm{w.d.}(M_N)<\infty$. So, $M_N$ is flat. Let $I$ be a finitely generated proper ideal. Then either $I\nsubseteq N$ and $I=Ra$ for some $a\in R\setminus N$, or $I\subseteq N$ and $I\cong (R/N)^n$ for some integer $n\geq 1$.

In the first case we have $\mathrm{w.d.}(R/I)=1$ since $a$ is a regular element, whence $\mathrm{Tor}_2^R(R/I,M)=0$. In the other case we have the following exact sequence:
\[0\rightarrow NF\rightarrow F\rightarrow I\rightarrow 0,\] 
where $F$ is a free $R$-module of rank $n$.  We deduce that the sequence 
\[0\rightarrow NF_N\rightarrow F_N\rightarrow I_N \rightarrow 0\quad\mathrm{is\ exact}.\]
We have $\mathrm{Tor}^Q_1(I_N,M_N)=0.$ Since $N$ is a $Q$-module, $NF\cong NF_N$. So,
in the following commutative diagram
\[\begin{array}{ccccccc}
0 &\rightarrow & \mathrm{Tor}^R_1(I,M) & \rightarrow & NF\otimes_RM & \rightarrow & F\otimes_RM \\
& & & & \downarrow & & \downarrow \\
& & 0 & \rightarrow & NF_N\otimes_QM_N & \rightarrow & F_N\otimes_QM_N
\end{array}\]
the left vertical map is an isomorphism. Then $NF\otimes_RM  \rightarrow  F\otimes_RM$ is a monomorphism. We successively deduce that $\mathrm{Tor}^R_1(I,M)=0$ and $\mathrm{Tor}^R_2(R/I,M)=0$. Hence $\mathrm{w.d.}(M)\leq 1$.
\end{proof}

´

\end{document}